\documentclass[12pt, reqno]{amsart}
\usepackage{amsmath, amsthm, amscd, mathrsfs, amsfonts, amssymb, graphicx}
\usepackage[bookmarksnumbered, colorlinks, plainpages]{hyperref}

\newtheorem{theorem}{Theorem}[section]
\newtheorem{lemma}[theorem]{Lemma}

\newtheorem{corollary}[theorem]{Corollary}
\theoremstyle{definition}

\theoremstyle{remark}
\newtheorem{remark}[theorem]{Remark}

\numberwithin{equation}{section}

\begin{document}
\title[A General Double Inequality]{A General Double Inequality Related to Operator Means and Positive Linear Maps}
\author[R. Kaur, M. Singh, J.S. Aujla and M. S. Moslehian]{Rupinderjit Kaur, Mandeep Singh, Jaspal Singh Aujla, M.S. Moslehian}
\address{Rupinderjit Kaur and Mandeep Singh: Department of Mathematics, Sant Longowal Institute of Engineering and
Technology, Longowal-148106, Punjab, India}
\email{rupinder\_grewal\_86@yahoo.co.in} \email{msrawla@yahoo.com}
\address{Jaspal Singh Aujla: Department of Mathematics, National Institute of
Technology, Jalandhar-144011, Punjab, India}
\email{aujlajs@nitj.ac.in and aujlajs@yahoo.com}
\address{Mohammad Sal Moslehian: Department of Pure Mathematics, Center of Excellence in
Analysis on Algebraic Structures (CEAAS), Ferdowsi University of
Mashhad, P.O. Box 1159, Mashhad 91775, Iran.}
\email{moslehian@ferdowsi.um.ac.ir and moslehian@member.ams.org}



\begin{abstract}
Let $A,B\in \mathbb{B}(\mathscr{H})$ be such that $0<b_{1}I \leq A
\leq a_{1}I$ and $0<b_{2}I \leq B \leq a_{2}I$ for some scalars
$0<b_{i}< a_{i},\;\; i=1,2$ and
$\Phi:\mathbb{B}(\mathscr{H})\rightarrow\mathbb{B}(\mathscr{K})$
 be a positive linear map. We show that for any operator mean $\sigma$ with the representing function $f$, the double inequality
$$
\omega^{1-\alpha}(\Phi(A)\#_{\alpha}\Phi(B))\le
(\omega\Phi(A))\nabla_{\alpha}\Phi(B)\leq
\frac{\alpha}{\mu}\Phi(A\sigma B)
$$
holds, where
$\mu=\frac{a_{1}b_{1}(f(b_{2}a_{1}^{-1})-f(a_{2}b_{1}^{-1}))}{b_{1}b_{2}-a_{1}a_{2}},~$
$\nu=\frac{a_{1}a_{2}f(b_{2}a_{1}^{-1})-b_{1}b_{2}f(a_{2}b_{1}^{-1})}{a_{1}a_{2}-b_{1}b_{2}},~$
$\omega=\frac{\alpha \nu}{(1-\alpha)\mu}$ and $\#_{\alpha}$
($\nabla_{\alpha}$, resp.) is the weighted geometric (arithmetic,
resp.) mean for $\alpha \in (0,1)$.

 As applications, we present several generalized operator
inequalities including Diaz--Metcalf and reverse Ando type
inequalities. We also give some related inequalities involving
Hadamard product and operator means.
\end{abstract}
\maketitle

\section{Introduction}

In what follows, $\mathbb{B}(\mathscr{H}) $ denotes the
$C^{*}$-algebra of all bounded linear operators acting on a Hilbert
space $(\mathscr{H}, \langle\cdot,\cdot \rangle)$ and $I$ stands for
the identity operator. A selfadjoint operator
$A\in\mathbb{B}(\mathscr{H})$ is said to be positive (strictly
positive, resp.) if $\langle A\xi,\xi \rangle\geq 0$\;($\langle
A\xi,\xi \rangle>0, \xi \ne 0$, resp.) for all $\xi\in\mathscr{H}$
and we write $A\geq 0$\, ($A>0$, resp.). For selfadjoint operators
$A,B \in \mathbb{B}(\mathscr{H})$, by
$A\geq B$\; ($A>B$, resp.) we mean $ A- B\geq 0$ \; ($A-B>0$, resp.). A linear map
$\Phi:\mathbb{B}(\mathscr{H}) \rightarrow \mathbb{B}(\mathscr{K})$
between $C^*$-algebras is called positive (strictly positive, resp.)
if it maps positive (strictly positive, resp.) operators into
positive (strictly positive, resp.) operators and is said to be
unital if it maps identity operator to identity operator in the
corresponding $C^{*}$-algebra.

By an operator monotone function, we mean a continuous real-valued
function $f$ defined on an interval $J$ such that $A\geq B$ implies
$f(A)\geq f(B)$ for all self adjoint operators $A, B $ with spectra
in $J$. Some structure theorems on operator monotone functions can
be found in \cite{FHPS}.

The axiomatic theory for operator means for
pairs of positive operators have been developed by Kubo and Ando
\cite{KA}. A binary operation $\sigma$ defined on the set of
strictly positive operators is called an operator mean provided that
\vspace{-1.0mm}
\begin{enumerate}
\item[(i)] $I\sigma I=I$;
\item [(ii)] $C^*(A\sigma B)C\leq(C^*AC)\sigma(C^*BC)$;
\item [(iii)] $A_n\downarrow A$ and $B_n\downarrow B$ imply $(A_n\sigma B_n)\downarrow A\sigma B$, where $A_n\downarrow A$ means that $A_1\geq A_2\geq \cdots$ and $A_n\to A$ as $n\to \infty$ in the strong operator topology;
\item [(iv)] $A \leq B$ and $C \leq D$ imply that $A\sigma C \leq B\sigma D$\,.
\end{enumerate}
\vspace{-1.0mm}
There exists an affine order isomorphism between the class of
operator means and the class of positive operator monotone functions
$f$ defined on $(0,\infty)$ with $f(1)=1$ via $f(t)I=I\sigma
(tI)\,\,(t> 0)$. In addition, $A \sigma B =
A^{1/2}f(A^{-1/2}BA^{-1/2})A^{1/2}$ for all strictly positive
operators $A, B$. The operator monotone function $f$ is called the
representing function of $\sigma$. Using a limit argument by
$A_\varepsilon=A+\varepsilon I$, one can extend the definition of
$A\sigma B$ to positive operators as well. The operator means
corresponding to operator monotone functions $(1-\alpha)+\alpha t$
and $t^{\alpha}$ with $0\leq \alpha \leq 1$ are called weighted arithmetic
and weighted geometric means and are denoted by $\nabla_{\alpha}$
and $\#_{\alpha}$, respectively.
In particular $\nabla_{1/2}$ and $\#_{1/2}$ or simply written as $\nabla$ and $\#$ are called arithmetic and geometric mean, respectively.\\

In this paper we establish a general double inequality involving
operator means and positive linear maps, which unifies and includes
the recent results of \cite{MNS, S} concerning Diaz--Metcalf and
reverse Ando type inequalities. We also give some related
inequalities involving Hadamard product and operator means.

\section{Main Result}
We start this section with our main result. It extends the reverse
Ando's inequality presented in \cite{S} and some Diaz--Metcalf type
inequalities in \cite{MNS} as we see in the sequel.

\begin{theorem}\label{main} Let $A,B\in \mathbb{B}(\mathscr{H})$ be
such that $0<b_{1}I \leq A \leq a_{1}I$ and $0<b_{2}I \leq B \leq
a_{2}I$ for some scalars $0<b_{i}< a_{i},\;\; i=1,2$ and
$\Phi:\mathbb{B}(\mathscr{H})\rightarrow\mathbb{B}(\mathscr{K})$
 be a positive linear map. Then for any operator mean $\sigma$ with the representing function $f$, the following double inequality holds:
\begin{equation}\label{(1)}
\omega^{1-\alpha}(\Phi(A)\#_{\alpha}\Phi(B))\le
(\omega\Phi(A))\nabla_{\alpha}\Phi(B)\leq
\frac{\alpha}{\mu}\Phi(A\sigma B)\,,
\end{equation}
where
$\mu=\frac{a_{1}b_{1}(f(b_{2}a_{1}^{-1})-f(a_{2}b_{1}^{-1}))}{b_{1}b_{2}-a_{1}a_{2}},~$
$\nu=\frac{a_{1}a_{2}f(b_{2}a_{1}^{-1})-b_{1}b_{2}f(a_{2}b_{1}^{-1})}{a_{1}a_{2}-b_{1}b_{2}},~$
$\omega=\frac{\alpha \nu}{(1-\alpha)\mu}$ and $\alpha \in (0,1).$
\end{theorem}
\begin{proof}
The conditions $0<b_{1}I\leq A\leq a_{1}I$ and $0<b_{2}I\leq B\leq
a_{2}I$ implies that $0<b_{2}a_{1}^{-1}A\leq b_2I\leq B \leq
a_2I\leq a_{2}b_{1}^{-1}A$. Consequently,
$0<b_{2}a_{1}^{-1}<a_{2}b_{1}^{-1}.$ The function
$f:(0,\infty)\rightarrow (0,\infty)$ being operator monotone is
strictly increasing and concave \cite[Corollary 1.12]{FHPS}. Therefore
$x^{-1}f(x)$ is operator monotone decreasing (see \cite[Corollary
1.14]{FHPS}). This implies that $\mu$ and $\nu$ are positive. In
fact,
$$\mu=\frac{1}{a_2b_1^{-1}-b_2a_1^{-1}}\left(f(a_{2}b_{1}^{-1})-f(b_{2}a_{1}^{-1})\right)>0$$
and
$$\nu=\frac{a_2b_2}{a_1a_2-b_1b_2}\left(b_2^{-1}a_1f(b_2a_1^{-1})-a_2^{-1}b_1f(a_2b_1^{-1})\right)>0\,.$$
The first inequality in \eqref{(1)} follows on using the weighted
arithmetic-geometric mean inequality $\omega^{1-\alpha}(X\# _\alpha
Y)=(\omega X)\# _\alpha Y\leq (\omega X)\nabla_\alpha Y$.

For the second inequality in \eqref{(1)}, consider $\mu
t+\nu$. Note that $f(t)$ and the line $\mu t+\nu$ intersect at the
points $(a_{2}b_{1}^{-1},f(a_{2}b_{1}^{-1}))$ and
$(b_{2}a_{1}^{-1},f(b_{2}a_{1}^{-1})).$ Thus, since $f(t)$ is
concave \cite[Corollary 1.12]{FHPS}, we see that
\begin{equation}\label{(2)}
\mu t+\nu \leq f(t)
\end{equation}
for all $t\in[b_{2}a_{1}^{-1},a_{2}b_{1}^{-1}]$. So,
$$\alpha t+(1-\alpha)\left[\frac{\alpha \nu}{(1-\alpha)\mu}\right] \leq \frac{\alpha}{\mu}f(t)$$
for all $\alpha \in (0,1).$ Hence, since $b_2a_1^{-1}I \leq
A^{-1/2}BA^{-1/2} \leq a_2b_1^{-1}I$, we obtain
$$\alpha A^{-1/2}BA^{-1/2}+(1-\alpha)\left[\frac{\alpha \nu}{(1-\alpha)\mu}\right] I\leq \frac{\alpha}{\mu}f(A^{-1/2}BA^{-1/2})\,,$$
which implies
\begin{equation}\label{(3)}
\alpha B+(1-\alpha)\left[\frac{\alpha \nu}{(1-\alpha)\mu}\right]A \leq \frac{\alpha}{\mu}A\sigma B.
\end{equation}
Since $\Phi$ is positive and linear, \eqref{(3)} yields
$$ (\omega\Phi(A))\nabla_{\alpha}\Phi(B)\leq \frac{\alpha}{\mu}\Phi(A\sigma B).$$
\end{proof}
\begin{remark}
The condition $0<b_{i}< a_{i},\;\; i=1,2$ in Theorem \ref{main} can
be replaced by either $0<b_{1}< a_{1}\,{\rm and}\, 0<b_{2}\leq a_{2}$ or
$0<b_{1}\leq a_{1}\,{\rm and}\, 0<b_{2}< a_{2}$.
\end{remark}
\section{Applications}

Now we present several applications of our main Theorem \ref{main}.

\subsection{Diaz--Metcalf type inequalities}\label{msm}

Taking $\alpha=1/2$ and $\sigma=\#$, so that $f(t)=\sqrt{t},$ in the
second inequality of \eqref{main}, we get the following
Diaz--Metcalf type inequality of the second type:
\begin{corollary} \cite[Theorem 2.1]{MNS}
Let $A, B \in \mathbb{B}(\mathscr{H})$ be positive invertible
operators such that $m_1^2I \leq A \leq M_1^2I$ and $m_2^2I\leq
B\leq M_2^2I$ for some positive real numbers $m_1<M_1$ and $m_2
<M_2$ and $\Phi: \mathbb{B}(\mathscr{H}) \to
\mathbb{B}(\mathscr{K})$ be a positive linear map. Then
\begin{eqnarray*}
\frac{M_2m_2}{M_1m_1}\Phi(A)+\Phi(B) \leq
\left(\frac{M_2}{m_1}+\frac{m_2}{M_1}\right)\Phi(A\sharp B)\,.
\end{eqnarray*}
\end{corollary}
If $m^2A \leq B \leq M^2A$ for some positive real numbers $m<M$,
then by considering $\Psi(C)=\Phi(A^{1/2}CA^{1/2})$ and noting that
$m^2I\leq A^{-1/2}BA^{-1/2}\leq M^2I$ and $1I\leq I\leq 1I$ we
obtain the following inequality:
$$Mm\Psi(I)+\Psi(A^{-1/2}BA^{-1/2})\leq (M+m)\Psi(I\#
A^{-1/2}BA^{-1/2})\,.$$ Therefore we reach to the following
Diaz--Metcalf type inequality of the first type:
\begin{corollary} \cite[Theorem 2.1]{MNS}
Let $A, B \in \mathbb{B}(\mathscr{H})$ be positive invertible
operators such that $m^2A \leq B \leq M^2A$ for some positive real
numbers $m<M$ and $\Phi: \mathbb{B}(\mathscr{H}) \to
\mathbb{B}(\mathscr{K})$ be a positive linear map. Then
\begin{eqnarray*}
Mm\Phi(A)+\Phi(B) \leq (M+m)\Phi(A\sharp B).
\end{eqnarray*}
\end{corollary}


\subsection{Inequalities complementary to Ando's inequality}
Ando's inequality \cite{A} states that if $A, B \in
\mathbb{B}(\mathscr{H})$ are positive operators, $\alpha\in[0,1]$
and $\Phi$ is a positive linear map, then
$$\Phi(A\#_{\alpha} B) \leq \Phi(A) \#_{\alpha}\Phi(B)\,.$$
The following result is an additive reverse of the second type of
this inequality:

\begin{corollary} Let $A,B$, $\Phi$ and $\mu ,\nu$ and $\alpha $ be as in Theorem \ref{main}. Then
\begin{eqnarray}\label{g2}
\Phi(A)\#_{\alpha}\Phi(B)-\Phi(A\#_{\alpha}B) \leq
\left((1-\alpha)(\mu
\alpha^{-1})^{\frac{\alpha}{\alpha-1}}-\nu\right)\Phi(A)\,.
\end{eqnarray}
In particular, when $\alpha=1/2.$
\begin{eqnarray}\label{com2}
\Phi(A)\#\Phi(B) -\Phi(A\# B)\leq
\left(\frac{1}{4\mu}-\nu\right)\min\{a_{1},a_{2}\}I
\end{eqnarray}
whenever $\Phi(I)\leq I$.
\end{corollary}
\begin{proof} From inequality \eqref{(2)} we get
$$
(1-\alpha)(\mu \alpha^{-1})^{\frac{\alpha}{\alpha-1}}+\alpha(\mu
\alpha^{-1}) t \leq f(t)-(\nu+(\alpha-1)(\mu
\alpha^{-1})^{\frac{\alpha}{\alpha-1}}).
$$
As in the proof of Theorem \ref{main} the above inequality yields
\begin{align*}
(1-\alpha)(\mu \alpha^{-1})^{\frac{\alpha}{\alpha-1}}\Phi(A)&+\alpha
(\mu
\alpha^{-1})\Phi(B)\\
&\leq \Phi(A\sigma B)-(\nu+(\alpha-1)(\mu
\alpha^{-1})^{\frac{\alpha}{\alpha-1}})\Phi(A).
\end{align*}
Now on using weighted arithmetic-geometric mean inequality on the
left hand side of the above inequality we get
\begin{align*} \Phi(A)\#_{\alpha}\Phi(B)
&\leq (1-\alpha)\left((\mu
\alpha^{-1})^{\frac{1}{\alpha-1}}\right)^\alpha\Phi(A)+\alpha
\left((\mu \alpha^{-1})^{\frac{1}{\alpha-1}}\right)^{\alpha-1}\Phi(B) \\
&\leq \Phi(A\sigma B)-\left(\nu+(\alpha-1)(\mu
\alpha^{-1})^{\frac{\alpha}{\alpha-1}}\right)\Phi(A)\,.
\end{align*}
Taking $\sigma =\#_{\alpha}$ and using the given conditions on $A,B$
we get the desired inequality \eqref{g2}. Inequality \eqref{com2}
immediately follows from \eqref{g2} for $\alpha=1/2$ because both
$\Phi(A)\#\Phi(B)$ and $\Phi(A\#B)$ are symmetric with respect to
$A$ and $B$.
\end{proof}
Using the technique in
Subsection \ref{msm} we can get the following
reverse Ando inequality of the second type:
\begin{corollary} \cite[Theorem 1]{S}
Let $A, B \in \mathbb{B}(\mathscr{H})$ be positive invertible
operators such that $mA \leq B \leq MA$ for some positive real
numbers $m<M$ and $\Phi: \mathbb{B}(\mathscr{H}) \to
\mathbb{B}(\mathscr{K})$ be a positive linear map. Then for $\alpha \in (0,1),$
\begin{eqnarray*}
\Phi(A)\#_{\alpha}\Phi(B)-\Phi(A\#_{\alpha}B) \leq \left[
(1-\alpha)\left(\frac{M^\alpha-m^\alpha}{\alpha(M-m)}\right)^{\frac{\alpha}{\alpha-1}}-\frac{Mm^\alpha-mM^\alpha}{M-m}\right]\Phi(A)\,.
\end{eqnarray*}
\end{corollary}
\subsection{Shisha--Mond and Kalmkin--McLenaghan inequalities}
Now, we present a Kalmkin--McLenaghan type inequality due
to Seo \cite{S}:

\begin{theorem} Let $A,B$, $\Phi$ and $\mu ,\omega $ be as in Theorem
\ref{main}. Then
$$\Phi(A\sigma B)^{-1/2}\Phi(B)\Phi(A\sigma B)^{-1/2}-\Phi(A\sigma B)^{1/2}\Phi(A)^{-1}\Phi(A\sigma B)^{1/2}\leq \left(\mu^{-1}-2\sqrt{\omega}\right)I.$$
\end{theorem}
\begin{proof} Now from \eqref{(1)} for $\alpha=1/2$ we get
$$\Phi(A\sigma B)^{-1/2}\Phi(B)\Phi(A\sigma B)^{-1/2}+ \omega \Phi(A\sigma B)^{-1/2}
\Phi(A)\Phi(A\sigma B)^{-1/2}\leq \mu^{-1}I\,,$$ whence
\begin{align*}
&\Phi(A\sigma B)^{-1/2}\Phi(B)\Phi(A\sigma B)^{-1/2}-\Phi(A\sigma
B)^{1/2}\Phi(A)^{-1}\Phi(A\sigma B)^{1/2}\\
&\leq \mu^{-1}I-\omega \Phi(A\sigma B)^{-1/2} \Phi(A)\Phi(A\sigma
B)^{-1/2}-\Phi(A\sigma B)^{1/2}
\Phi(A)^{-1}\Phi(A\sigma B)^{1/2}\\
&\leq
(\mu^{-1}-2\sqrt{\omega})I\\
&\qquad -\left(\sqrt{\omega}\left(\Phi(A\sigma B)^{-1/2}
\Phi(A)\Phi(A\sigma B)^{-1/2}\right)^{1/2}-\left(\Phi(A\sigma
B)^{1/2}
\Phi(A)^{-1}\Phi(A\sigma B)^{1/2}\right)^{1/2}\right)^2\\
&\leq (\mu^{-1}-2\sqrt{\omega})I\,,
\end{align*}
which is a generalized operator Shisha--Mond inequality.
\end{proof}
If $\sigma$ is taken to be $\#$, we get the following Operator
Shisha--Mond inequality:
\begin{corollary} \cite[Theorem 2.1]{MNS}
Let $A, B \in \mathbb{B}(\mathscr{H})$ be positive invertible
operators such that $m^2A \leq B \leq M^2A$ for some positive real
numbers $m<M$ and $\Phi: \mathbb{B}(\mathscr{H}) \to
\mathbb{B}(\mathscr{K})$ be a positive linear map. Then
\begin{eqnarray*}
\Phi(A\sharp B)^{\frac{-1}{2}}\Phi(B)\Phi(A\sharp B)^{\frac{-1}{2}}-
\Phi(A\sharp B)^{\frac{1}{2}}\Phi(A)^{-1}\Phi(A\sharp
B)^{\frac{1}{2}} \leq (\sqrt{M}-\sqrt{m})^{2} I\,.
\end{eqnarray*}
\end{corollary}
Using the technique in Subsection \ref{msm} we can get the following
Kalmkin--McLenaghan inequality:

\begin{corollary} \cite[Theorem 2.1]{MNS}
Let $A, B \in \mathbb{B}(\mathscr{H})$ be positive invertible
operators such that $m_1^2I \leq A \leq M_1^2I$ and $m_2^2I\leq
B\leq M_2^2I$ for some positive real numbers $m_1<M_1$ and $m_2
<M_2$ and $\Phi: \mathbb{B}(\mathscr{H}) \to
\mathbb{B}(\mathscr{K})$ be a positive linear map. Then
\begin{eqnarray*}
\Phi(A\sharp B)^{\frac{-1}{2}}\Phi(B)\Phi(A\sharp B)^{\frac{-1}{2}}-
\Phi(A\sharp B)^{\frac{1}{2}}\Phi(A)^{-1}\Phi(A\sharp
B)^{\frac{1}{2}}\leq
\left(\sqrt{\frac{M_2}{m_1}}-\sqrt{\frac{m_2}{M_1}}\right)^{2} I\,.
\end{eqnarray*}
\end{corollary}
\subsection{Ozeki--Izumino--Mori--Seo type inequalities}
Now we present a generalized operator Ozeki--Izumino--Mori--Seo type
inequality. To achieve it, we need the following lemma, which is
helpful in proving many operator inequalities.

\begin{lemma}\label{l1} Let $\Phi$ be a unital positive linear map
on
 $\mathbb{B}(\mathscr{H}),$ $A\in\mathbb{B}(\mathscr{H})$
 is selfadjoint with $bI\leq A\leq aI$. Then
 $$\Phi(A^{2})-\Phi(A)^{2} \leq \frac{1}{4} (a-b)^{2}I.$$
\end{lemma}
\begin{proof} The proof is an easy consequence of both facts
$$\Phi(A^{2})
 -\Phi(A)^{2}=\Phi(|A-\alpha I|^{2})-|\Phi(A-\alpha I)|^{2}\,\,\,\,\,\,(\alpha\in \mathbb{R})$$ and
 $$\frac{1}{4} (a-b)^{2}I-\left|A-\frac{a+b}{2}I \right|^{2}=(aI-A)(A-bI)\geq 0.$$
\end{proof}
\begin{theorem} Suppose that $\Phi:\mathbb{B}(\mathscr{H}) \rightarrow \mathbb{B}(\mathscr{K})$
 is a strictly positive linear map with $\Phi(I)\leq I.$
 Assume that $A,B\in\mathbb{B}(\mathscr{H})$ are such that $0<b_{1}I\leq A\leq a_{1}I$ and
 $0<b_{2}I\leq B\leq a_{2}I.$ Then
\begin{align*}
\Phi(A)^{1/2}\Phi(|A^{-1/2}(A\sigma
B)|^{2})\Phi(A)^{1/2}&-\left|\Phi(A)^{-1/2}\Phi(A\sigma
B)\Phi(A)^{1/2}\right|^{2}\\
&\leq\frac{a_{1}^{2}}{4} (f(a_{2}b_{1}^{-1})-f(b_{2}a_{1}^{-1})
)^{2}I\,.
\end{align*}
\end{theorem}
\begin{proof}
 Consider the strictly positive linear map $\Psi :\mathbb{B}(\mathscr{H})\rightarrow \mathbb{B}(\mathscr{K}) $ defined by
 $$\Psi(C)=\Phi(A)^{-1/2}\Phi(A^{1/2}CA^{1/2})\Phi(A)^{-1/2}.$$
Clearly $\Psi(I)=I$. Utilizing Lemma \ref{l1}, we obtain
 \begin{equation}\label{(7)}
 \Psi(C^{2})-\Psi(C)^{2}\leq\frac{1}{4} (a-b)^{2}I
 \end{equation}
 for all $C$ with $0<bI\leq C\leq aI.$ Put $C:=f(A^{-1/2}BA^{-1/2}).$ The given conditions on $A,B$ imply that
 $$
 b_2a_1^{-1}I\leq A^{-1/2}BA^{-1/2}\leq a_2b_1^{-1}I.
 $$
 The operator monotonicity of $f$ then yields
 $$
 0< bI=f(b_2a_1^{-1})I\leq C= f(A^{-1/2}BA^{-1/2})\leq f(a_2b_ 1^{-1}) I = aI.
 $$
 Thus inequality \eqref{(7)} gives rise to
 \begin{align*}
\Phi(A)^{-1/2}\Phi(|A^{-1/2}(A\sigma
B)|^{2})\Phi(A)^{-1/2}&-\left(\Phi(A)^{-1/2}\Phi(A\sigma
 B)\Phi(A)^{-1/2}\right)^{2}\\
&\leq\frac{1}{4} (f(a_{2}b_{1}^{-1})-f(b_{2}a_{1}^{-1}) )^{2}I.
\end{align*}
 Now on pre and post multiplying by $\Phi(A),$ we obtain
\begin{align*}
\Phi(A)^{1/2}\Phi(|A^{-1/2}(A\sigma
B)|^{2})\Phi(A)^{1/2}&-\left|\Phi(A)^{-1/2}\Phi(A\sigma
B)\Phi(A)^{1/2}\right|^{2}\\
&\leq \frac{a_{1}^{2}}{4} (f(a_{2}b_{1}^{-1})-f(b_{2}a_{1}^{-1})
 )^{2}I\,.
 \end{align*}
\end{proof}
\begin{corollary}\label{c1} Let
 $A,B\in \mathbb{B}(\mathscr{H})$
 be such that $0< b_{1}I\leq A \leq a_{1}I$ and $0< b_{2}I\leq B \leq a_{2}I$, respectively. Then for any unit vector
 $x\in {\mathscr H}$,
 $$\langle Ax, x \rangle \langle |A^{-1/2}(A\sigma B)|^{2}x, x \rangle-\langle A
 \sigma B x, x \rangle^{2}\leq \frac{a_{1}^{2}}{4}
 (f(a_{2}b_{1}^{-1})-f(b_{2}a_{1}^{-1}
 ))^{2}.$$
 \end{corollary}
 \begin{proof} The corollary follows on considering the unital positive linear map $\Phi $ on $\mathbb{B}(\mathscr{H})$ given by $\Phi (C)= ~<Cx,x>.$
 \end{proof}
\begin{corollary}\cite[Theorem 4.5]{IMS} Let
 $A,B\in \mathbb{B}(\mathscr{H})$
 be such that $0< b_{1}I\leq A \leq a_{1}I$ and
 $0< b_{2}I\leq B \leq a_{2}I$, respectively. Then for any unit vector
 $x\in {\mathscr H}$,
 $$\langle Ax, x \rangle \langle Bx, x \rangle-\langle A
 \# B x, x \rangle^{2}\leq\left( \frac{\sqrt{a_{1}a_{2}} -\sqrt{b_{1}b_{2}}}{2}\right)^{2} \min \{a_{1}b_{1}^{-1},a_{2}b_{2}^{-1}\}\,.$$
 \end{corollary}
 \begin{proof} Replacing $\sigma$ by $\#$ in Corollary \ref{c1}, we get the desired result.
 \end{proof}

\subsection{Greub--Rheinboldt type inequality}

Greub--Rheinboldt \cite{GR} showed that if $A \in
\mathbb{B}(\mathscr{H})$ be such that and $0<mI\leq A\leq MI$, then
\begin{eqnarray*}
\langle Ax, x\rangle \langle A^{-1}x, x\rangle  \leq
\frac{(M+m)^2}{4mM}\quad (x\in \mathscr{H}, \|x\|=1)\,.
\end{eqnarray*}
The first consequence of our main result is a generalized
Greub-Rheinboldt inequality.

\begin{corollary} If $A \in
\mathbb{B}(\mathscr{H})$ be such that $0<mI\leq A\leq MI$, then for any
$0<\alpha<1$, any operator mean $\sigma$ and any positive linear map
$\Phi$ the following
Greub--Rheinboldt type inequality is valid:\\
$$\Phi(A)\#_{\alpha}\Phi(A^{-1})\leq \frac{\alpha}{\mu\omega^{1-\alpha}}\Phi(A\sigma A^{-1})\,,$$
where $\mu$ and $\omega$ are determined as in Theorem \ref{main}
with $a_1=M, b_1=m, a_2=m^{-1}$ and $b_2=M^{-1}$.
\end{corollary}
\begin{proof} Using \eqref{(1)} with $B=A^{-1}$ we get the desired inequality.
\end{proof}

\section{Inequalities involving Hadamard product and operator means}
In this section we present several inequalities involving Hadamard
product and operator means.

If $U$ is the isometry of $H$ into $H\otimes H$ given by
$Ue_{n}=e_{n}\otimes e_{n}$, where $\{e_{n}\}$ is a fixed
orthonormal basis of $H,$ then the Hadamard product $A\circ B$ of
(bounded) operators $A$ and $B$ on $H$ for $\{e_{n}\}$ is expressed
by
$$ A\circ B=U^{*}(A\otimes B) U.$$
A real valued continuous function $f$ is called supermultiplicative
(submultiplicative, resp.) if $f(xy) \geq f(x)f(y)$ ($f(xy) \leq
f(x)f(y)$, resp.).
\begin{theorem} Let $A,B,C, D\in
\mathbb{B}(\mathscr{H})$ be such that $b_{1}I \leq A \leq a_{1}I$,
$b_{2}I \leq B \leq a_{2}I$, $b_{3}I \leq C \leq a_{3}I$ and $b_{4}I
\leq D \leq a_{4}I$ for some scalars $0<b_{i}< a_{i},\;\; i=1,
\cdots, 4.$ Then for any operator mean $\sigma$, whose representing
function $f$ is submultiplicative, the following generalized
inequalities hold for $\alpha \in (0,1)$:
\begin{itemize}
\item[(i)] $\big(\omega(A\circ B)\big) \nabla_{\alpha} (C\circ D)\leq \frac{\alpha}{\mu}((A \sigma C)\circ(B\sigma D))$.\\

\item[(ii)] $\omega^{1-\alpha}\big((A\circ B)\#_{\alpha}(C\circ D)\big)\leq \frac{\alpha}{\mu}((A\sigma C)\circ(B\sigma D))$.\\

\item[(iii)]$(A\circ B) \#_{\alpha} (C\circ D)-((A\sigma
C)\circ(B\sigma D))
\leq (\frac{\alpha}{\mu} \omega^{\alpha-1}-1) a_{1}a_{2}f(a_{3}a_{4}b_{1}^{-1}b_{2}^{-1})I$.\\
In particular,
\begin{align*}
(A\circ B) \# (C\circ D)&-((A\# C)\circ(B\# D))\\
&\leq (\frac{1}{2\mu \sqrt{\omega}}-1)
\min\left\{a_{1}a_{2}f(a_{3}a_{4}b_{1}^{-1}b_{2}^{-1}),a_{3}a_{4}(a_{1}^{-1})a_{2}^{-1}b_{3}b_{4}^{1/2}\right\}I\,.
\end{align*}
\item[(iv)] $(A\circ B) \#_{\alpha} (C\circ D)-((A\#_{\alpha}
C)\circ(B\#_{\alpha} D))
\leq \left((1-\alpha)(\mu
\alpha^{-1})^{\frac{\alpha}{\alpha-1}}-\nu\right) a_{1}a_{2}I$.\\
In particular,\\
$(A\circ B) \# (C\circ D)-((A\# C)\circ(B\# D))
\leq (\frac{1}{4\mu}-\nu ) \min\left\{a_{1}a_{2},a_{3}a_{4}\right\}I$.\\
\item[(v)]\begin{align*}((A\sigma C)&\circ(B\sigma D))^{-1/2}(C\circ D)((A\sigma C)\circ(B\sigma D))^{-1/2}\\
&-((A\sigma C)\circ(B\sigma D))^{1/2}(A\circ B)^{-1}((A\sigma
C)\circ(B\sigma D))^{1/2}\leq (\mu^{-1}-2\sqrt{\omega})I\,,
\end{align*}
\end{itemize}
where
$\omega=\frac{\alpha(b_{1}b_{2}b_{3}b_{4}f(a_{3}a_{4}b_{1}^{-1}b_{2}^{-1})-a_{1}a_{2}a_{3}a_{4}f(b_{3}b_{4}a_{1}^{-1}a_{2}^{-1}))}
{(1-\alpha)a_{1}a_{2}b_{1}b_{2}(f(b_{3}b_{4}a_{1}^{-1}a_{2}^{-1})-f(a_{3}a_{4}b_{1}^{-1}b_{2}^{-1}))},
\mu=\frac{a_{1}a_{2}b_{1}b_{2}(f(b_{3}b_{4}a_{1}^{-1}a_{2}^{-1})-f(a_{3}a_{4}b_{1}^{-1}b_{2}^{-1}))}{b_{1}b_{2}b_{3}b_{4}-a_{1}a_{2}a_{3}a_{4}}$
and $\nu =\frac{(1-\alpha)\omega\mu}{\alpha}.$
\end{theorem}
\begin{proof} We have
\begin{align*}
(A\otimes B)^{1/2}(f(X)\otimes f(Y))(A\otimes B)^{1/2}\geq (A\otimes
B)^{1/2}f(X\otimes Y)(A\otimes B)^{1/2}\,.
\end{align*}
Taking $X=A^{-1/2}CA^{-1/2}$ and $Y=B^{-1/2}DB^{-1/2},$ we obtain
\begin{equation}\label{(8)}
(A\sigma C)\otimes(B\sigma D)\geq (A\otimes B)\sigma(C\otimes D).
\end{equation}
Also, simple arguments leads
to $a_{1}^{-1}a_{2}^{-1}b_{3}b_{4}A\otimes B \leq C\otimes D\leq b_{1}^{-1}b_{2}^{-1}a_{3}a_{4}A\otimes B,$ so on replacing $A$ by $A\otimes B$ and $B$ by $C\otimes D$
in Theorem
\ref{main} and taking $\Phi(Y)=U^*YU$, where $U$ an isometry
satisfying $U^{*} (A\otimes B)U=A\circ B$, and using \eqref{(8)} we
get (i) and (ii). The rest of the inequalities can be proved similarly.
\end{proof}

\begin{theorem} Let $A,B,C, D, \sigma$ and $f$ as in Theorem 4.1.Then, the following generalized inequalities
hold:
\begin{itemize}
\item[(i)] $ (A\circ B)\sigma(C\circ D)\leq\frac{1}{\omega}((A\sigma C)\circ(B\sigma D))$,\\

\item[(ii)] $((A\circ B)\sigma(B\circ D))-((A\sigma C)\circ(B\sigma D))\leq -g(t_{0})(A\circ B)\,,$
\end{itemize}
where $\omega=\frac{a_{1}a_{2}b_{1}b_{2}(f(a_{3}a_{4}b_{1}^{-1}b_{2}^{-1})-f(a_{1}^{-1}a_{2}^{-1}b_{3}b_{4}))}{(a_{1}a_{2}a_{3}a_{4}-b_{1}b_{2}b_{3}b_{4})f'(c)}$ $\mu= \frac{a_{1}a_{2}b_{1}b_{2}(f(a_{3}a_{4}b_{1}^{-1}b_{2}^{-1})-f(a_{1}^{-1}a_{2}^{-1}b_{3}b_{4}))}{a_{1}a_{2}a_{3}a_{4}-b_{1}b_{2}b_{3}b_{4}}$,\\ $\nu =\frac{a_{1}a_{2}a_{3}a_{4}(f(a_{1}^{-1}a_{2}^{-1}b_{3}b_{4})-b_{1}b_{2}b_{3}b_{4}f(b_{1}^{-1}b_{2}^{-1}a_{3}a_{4}))}{a_{1}a_{2}a_{3}a_{4}-b_{1}b_{2}b_{3}b_{4}},$ and $g(t)=\mu t+\nu-f(t)$ for
$c$ and $t_{0}$ some fixed points in $(a_{1}^{-1}a_{2}^{-1}b_{3}b_{4},a_{3}a_{4}b_{1}^{-1}b_{2}^{-1}).$
\end{theorem}
\begin{proof} It is known \cite{KSA} that if $A,B>0$ be such that $aA \geq B \geq bA$ for some scalars
$a \geq b>0$ and  $\Phi$ is a positive linear map, then for any
connection $\sigma$,
\begin{eqnarray}\label{R1}
\Phi(A) \sigma \Phi(B) \leq \frac {1}{\omega}\Phi(A\sigma B)
\end{eqnarray}
and
\begin{eqnarray}\label{R2}
\Phi(A) \sigma \Phi(B)-\Phi(A\sigma B) \leq
-g(t_{0})\Phi(A)\,,
\end{eqnarray} where
$\omega=\frac{f(a)-f(b)}{(a-b)f'(c)}$ for some fixed $c \in (b,a)$
and $g(t)=\mu t+\nu-f(t)$, $t_{0}$ a fixed point in $(b,a)$ with
$-g(t_{0})\geq 0$, $\mu= \frac {f(a)-f(b)} {a-b}, \nu =\frac {af(b)-bf(a)} {a-b}$ and $f(t)$ is the representing function of $\sigma.$

As in Theorem 4.1, replacing $A$ by $A\otimes B$ and $B$ by $C\otimes D,$ in \eqref{R1} and \eqref{R2} and using inequality \eqref{(8)}, we get

$$\Phi(A\otimes B)\sigma\Phi(C\otimes
D)\leq \frac{1}{\omega}\Phi((A\sigma C)\otimes(B\sigma D))$$ and
$$\Phi(A\otimes B)\sigma \Phi(C\otimes D)-\Phi((A\sigma C)\otimes(B\sigma
D))\leq -g(t_{0})(A\otimes B)\,.$$
Once again (i) and (ii) can be deduced by
considering $\Phi(Y)=U^*YU$, where $U$ an isometry satisfying $U^{*}
(A\otimes B)U=A\circ B$.
\end{proof}

\end{document}